\documentclass{tac}

\usepackage{amsmath,amssymb}

\usepackage[colorlinks=true]{hyperref}
\hypersetup{allcolors=[rgb]{0.1,0.1,0.4}}

\author{Pieter Hofstra and Martti Karvonen}

\thanks{}

\address{Department of Mathematics and Statistics, University of Ottawa\\
150 Louis-Pasteur Private, Ottawa, ON, Canada K1N 6N5\\[5pt]
}

\title{Inner automorphisms as $2$-cells}

\copyrightyear{2024}


\eaddress{martti.karvonen@uottawa.ca}

\usepackage{tikz,enumitem,xspace,cite}
\usetikzlibrary{matrix,arrows}
\tikzstyle{twocell}=[-implies,double equal sign distance]
\tikzstyle{shorttwocell}=[-implies,double equal sign distance,shorten >=7pt, shorten <=7pt]

\newcommand{\cat}[1]{\ensuremath{\mathbf{#1}}}
\newcommand{\CC}{\cat{C}}
\newcommand{\DD}{\cat{D}}

\newcommand{\JJ}{\cat{J}} %for indexing diagrams D:J->C
\newcommand{\Grp}{\cat{Grp}}
\newcommand{\Set}{\cat{Set}}
\newcommand{\Cat}{\cat{Cat}}
\newcommand{\op}{\ensuremath{{}^{\mathrm{op}}}}
\newcommand{\id}[1][]{\ensuremath{\mathrm{id}_{#1}}}
\newcommand{\adjoin}[1]{\mathopen{\langle}#1\mathclose{\rangle}} 
\newcommand{\ie}{\text{i.e.}\xspace}
\newcommand{\eg}{\text{e.g.}\xspace}
\DeclareMathOperator{\cod}{cod}

\DeclareMathOperator{\Aut}{Aut}
\DeclareMathOperator{\Ob}{Ob}
\DeclareMathOperator{\colim}{colim} 
\newcommand{\GG}{\mathcal{G}} %presheaf of grps
\newcommand{\HH}{\mathcal{H}} %another such
\newcommand{\ZZ}{\mathcal{Z}} %isotropy
\newcommand{\II}{\mathcal{I}}%the inclusion Grp\to\Cat
\newcommand{\UU}{\mathcal{U}}%the forgetful functor Grp\to\Set
\begin{document}

\maketitle	

\begin{abstract}
Abstract inner automorphisms can be used to promote any category into a $2$-category, and we study two-dimensional limits and colimits in the resulting $2$-categories. Existing connected colimits and limits in the starting category become two-dimensional colimits and limits under fairly general conditions. Under the same conditions, colimits in the underlying category can be used to build many notable two-dimensional colimits such as coequifiers and coinserters. In contrast, disconnected colimits or genuinely $2$-categorical limits such as inserters and equifiers and cotensors cannot exist  unless no nontrivial abstract inner automorphisms exist and the resulting $2$-category is locally discrete. We also study briefly when an ordinary functor can be extended to a $2$-functor between the resulting $2$-categories.
\end{abstract}

\section*{On Pieter} 

I, the second named author, worked with Pieter from when I arrived at the University of Ottawa in 2019 until his passing in 2022. We established (special cases of) the the main results of this work early in our collaboration, before we decided to put this work on a back burner and focus on~\cite{hofstrakarvonen:2d-isotropy}. A week before his death, I received Pieter's final comments on our draft of~\cite{hofstrakarvonen:2d-isotropy}, after which I completed that manuscript and then returned to this one with the current memorial volume in mind: while he never saw the final version, I hope that it is to his taste, especially since the natural setting for the theory was given by crossed modules studied in~\cite{funketal:crossedtoposes} instead of the special case we initially worked on. 

I really enjoyed the time I had working with Pieter, and he taught me a lot about categorical thinking. He was great at explaining things, asked insightful questions, and was also good company outside of work. I went on occasional bike rides with him and Phil Scott when most things were closed during the pandemic, seeing a side of them and of Ottawa that I would have otherwise missed. I remember these bike rides very fondly, despite the fact that I was mostly struggling to keep up while these two seasoned riders leisurely chatted away. 

\section{Introduction}

The category of groups can be viewed as a $2$-category in a nontrivial way, where $2$-cells $\psi\to \phi$ correspond to elements $g$ of the codomain such that $g\psi(-)g^{-1}=\phi$. This $2$-categorical structure is not just a curiosity but arises in useful group-theoretic constructions. For instance, the HNN-extension of a group $G$ along two subgroup inclusions $\psi,\phi\colon H\rightrightarrows G$ is often defined in terms of the concrete construction 
  \[G\adjoin{t}/\{t\psi(h)t^{-1}=\phi(h)|h\in H\},\]
but can readily be seen as the universal way of adding a $2$-cell $\psi\to\phi$, in other words, as the coinserter of $\psi$ and $\phi$.

This $2$-categorical structure is often viewed as arising from the fact that groups can be thought of as one-object categories, giving an embedding $\Grp\to\Cat$ of $1$-categories which we can promote to an embedding of $2$-categories by having $\Grp$ inherit its $2$-cells from $\Cat$. In this work, we take a different point of view: the $2$-cells between group homomorphisms come from inner automorphisms, \ie automorphisms induced by conjugating with a fixed element. In turn, these inner automorphisms can be characterized abstractly, resulting in a notion applicable to any category. 

To see how this works, note first that an element $g\in G$ gives us more than just the function $\alpha_{\id}\colon G\to G$ given by conjugation with $g$: For any $f\colon G\to H$, we get a function $\alpha_f\colon H\to H$ given by conjugation with $f(g)$.  Moreover, for any $h\colon H\to H'$ the square  
         \[\begin{tikzpicture}
    \matrix (m) [matrix of math nodes,row sep=2em,column sep=4em,minimum width=2em]
    {
     H & H\\
     H' & H'  \\};
    \path[->]
    (m-1-1) edge node [left] {$h$} (m-2-1)
           edge node [above] {$\alpha_f$} (m-1-2)
    (m-1-2) edge node [right] {$h$} (m-2-2)
    (m-2-1) edge node [below] {$\alpha_{hf}$}  (m-2-2);
  \end{tikzpicture}\] commutes.
Abstractly, this can be captured by saying that  $g$ induces a natural automorphism of the projection $G/\Grp\to\Grp$. A theorem by Bergman~\cite{bergman:inner} shows that any such natural automorphism is induced in this way.

These \emph{extended inner automorphisms} make sense in any category $\CC$, so that one can define the group $\ZZ(A)$ of them at an object $A\in \CC$ as the group of natural automorphisms of the projection $A/\CC\to\CC$. This results in the notion of (covariant) \emph{isotropy group} studied in~\cite{parker:IsotropyofGrothendieckToposes,hofstraetal:picard,funketal:higherisotropy,hofstraetal:isotropyofalgtheories,hofstrakarvonen:2d-isotropy}. The starting point of this work is that such extended inner automorphisms can be used to promote any category $\CC$ into a $2$-category $\CC_\ZZ$ by defining, for $f,g\colon A\rightrightarrows B$, $2$-cells $f\to g$ as ``extended inner automorphisms that take $f$ to $g$'', \ie as elements $\alpha\in \ZZ(B)$ such that $\alpha_{\id}f=g$. This gives another explanation for the $2$-category of groups, as $G\cong \ZZ(G)$ naturally in the group $G$. In fact, one can do this construction more generally: all that one needs is a copresheaf $\GG\colon \CC\to\Grp$ where $\GG(A)$ is thought of as ``abstract inner automorphisms of $A$'', equipped with suitably compatible maps $\gamma_A\colon \GG(A)\to \Aut(A)$ that let elements of $\GG(A)$ act as automorphisms on $A$. This results in the notion of a crossed $\CC$-module, studied in the dual case in~\cite[Section 5]{funketal:crossedtoposes}. 

In this work, we study two-dimensional limits and colimits in $2$-categories that arise in this manner. Our first result, Theorem~\ref{thm:connectedcolims} states that any connected colimit in the underlying category becomes a two-dimensional colimit. Moreover, when the presheaf of sets underlying the presheaf of groups is representable, many further pleasant properties hold as shown in Theorem~\ref{thm:positiveresults}: in particular, all limits in the underlying category are also two-dimensional limits, and if the starting category was finitely cocomplete, one can build arbitrary coinserters by mimicking the construction of HNN-extensions in $\Grp$. In a sense, these positive results are sharp, as evidenced by Theorems~\ref{thm:obstructionsforcolims},\ref{thm:limshavetrivial2cells} and Corollary~\ref{cor:obstructionslims}, which roughly speaking state that one cannot have more in the way of two-dimensional (co)limits unless $\GG$ is trivial. In this way, the two-dimensional (co)limit behavior of $\Grp$ holds in very general conditions. We conclude by discussing briefly when a functor $\CC\to\DD$ extends to a $2$-functor $\CC_\GG\to\DD_\HH$: in particular, we will show that $\GG$ extends canonically to a $2$-functor to groups. This also holds for its left adjoint $L$ whenever it exists, so that $L$ preserves not only ordinary colimits of groups, but also all two-dimensional colimits that exist in $\cat{Grp}$.

\section{Background}

\subsection{Two-dimensional (co)limits}

We refer the reader to~\cite{kelly:2-dlimits} for general background on (strict) $2$-categorical (co)limits and use this section merely to fix notation and terminology. A diagram in a $2$-category $\CC$ is given by a $2$-functor $D\colon\JJ\to \CC$. In general, we require weighted limits, with the weight being given by a $2$-functor $W\colon \JJ\to \Cat$. Given $D$ and $W$, the $W$-weighted limit of $D$ is an object $\lim^W D$ of $\CC$ whose universal property is given by isomorphisms
\[\CC(A,\lim^W D)\cong [\JJ,\Cat](W,\CC(A,D(-)) \] 
that are $2$-natural in $A$. This correspondence is often split into two parts corresponding the objects and morphisms of $\CC(A,\lim^W D)$. The one-dimensional universal property states that maps $A\to \lim^W D$ correspond to  natural transformations $W\to \CC(A,D(-))$. We will stretch the usual one-dimensional terminology here and refer to such a transformation as a cone on $A$, or a  $W$-weighted cone if we wish to be more precise. In particular, the identity on $\lim^W D$ corresponds to a natural transformation $W\to \CC(\lim^W D,D(-))$ that we will call the universal cone. The two-dimensional universal property then states that modifications between transformations  $W\to \CC(A,D(-))$ correspond to $2$-cells between the corresponding maps $A\to\lim^W D$. Strictly speaking, establishing that an object satisfies both the one- and two-dimensional universal properties is not quite enough to exhibit it as a weighted limit, as the required isomorphism should be one of categories and not just a bijection between their sets of objects and hom-sets. However, in this paper we will only focus on the one-dimensional and two-dimensional universal properties and omit checks of functoriality as routine.

In the sequel, we will often study whether an ordinary (\ie $1$-categorical) (co)limit in the underlying $1$-category of $\CC$ is also a $2$-categorical (co)limit in $\CC$: in such cases the weight $W$ is assumed to be constant at the terminal category $\cat{1}$ and suppressed from the notation.

The situation and notation for weighted colimits is dual, with the weight $W\colon \JJ\op\to \Cat$  now being contravariant, and the colimit (if it exists) denoted by $\colim^W D$. Again, we will speak of ($W$-weighted) cocones and of the universal cocone. 

\subsection{Crossed modules and isotropy}

In this section we recall some background on crossed modules and isotropy. We essentially follow the development in~\cite[Section 5]{funketal:crossedtoposes}, except that we will work with the dual notion and call these simply crossed modules instead crossed $\CC$-modules.

\begin{definition}[Definition 5.2 of \cite{funketal:crossedtoposes}] Let $\CC$ be a category. A \emph{crossed module} consists of 
  \begin{itemize}
    \item a copresheaf $\GG\colon\CC\to\Grp$ 
    \item group homomorphisms $\gamma_A\colon \GG(A)\to \Aut(A,A)$ for each object $A$ of $\CC$
  \end{itemize}
  such that 
  \begin{enumerate}
    \item for each $f\colon A\to B$ in $\CC$ and $\alpha\in \GG(A)$ we have $f\gamma_A(\alpha)=\gamma_B(\GG(f)\alpha)f$, \ie the square 
      \[\begin{tikzpicture}
    \matrix (m) [matrix of math nodes,row sep=2em,column sep=4em,minimum width=2em]
    {
     A & A  \\
     B & B  \\};
    \path[->]
    (m-1-1) edge node [left] {$f$} (m-2-1)
           edge node [above] {$\gamma_A(\alpha)$} (m-1-2)
    (m-1-2) edge node [right] {$f$} (m-2-2)
    (m-2-1) edge node [below] {$\gamma_B(\GG(f)\alpha)$}  (m-2-2);
  \end{tikzpicture}\]
  commutes.
    \item for each $\alpha,\beta\in \GG(A)$ we have $\GG(\gamma(\alpha))\beta=\alpha \beta \alpha ^{-1}$ (Peiffer identity).
  \end{enumerate}
\end{definition}

Crossed modules organize themselves into a $2$-category.

\begin{definition}[Definition 5.6 of \cite{funketal:crossedtoposes}]\label{def:crossedmod} A morphism of crossed modules 
    \[(\GG \colon \CC\to\Grp,\gamma)\to (\HH\colon \DD\to\Grp,\delta)\] 
consists of a functor $F\colon \CC\to\DD$ and of a natural transformation $\sigma\colon \GG\to \HH\circ F$ such that the square 
  \[\begin{tikzpicture}
    \matrix (m) [matrix of math nodes,row sep=2em,column sep=4em,minimum width=2em]
    {
     \GG(A) & \HH F(A) \\
      \Aut(A) & \Aut(FA)  \\};
    \path[->]
    (m-1-1) edge node [left] {$\gamma_A$} (m-2-1)
           edge node [above] {$\sigma$} (m-1-2)
    (m-1-2) edge node [right] {$\delta_{FA}$} (m-2-2)
    (m-2-1) edge node [below] {$F$}  (m-2-2);
  \end{tikzpicture}\]
commutes for each $A\in\CC$. For two parallel morphisms $(F,\sigma)$ and $(G,\theta)$ of crossed modules, a $2$-cell $(F,\sigma)\to (G,\theta)$ consists of a natural transformation $\tau\colon F\to G$ such that 
\[\theta=(\HH \tau)\circ \sigma\]
\end{definition}

We move to our main example of a crossed module.

\begin{definition}
 Let $\CC$ be a category and $X$ an object of $\CC$. 
Then the (covariant) \emph{isotropy group of $\CC$ at $X$} is the group $\ZZ(X)$ of natural automorphisms of the projection functor $P_X\colon X/\CC\to \CC$. 
\end{definition}

Explicitly, an automorphism 
\begin{equation}\label{eq:alpha} 
\alpha=(\alpha_f)_{f:X \to A} \in \ZZ(X)=_{\mathrm{def}} \Aut(P_X\colon X/\CC\to \CC) \end{equation}
 consists of an automorphism $\alpha_f\colon A\to A$ for each $f\colon X\to A$ such that the square
    \[\begin{tikzpicture}
    \matrix (m) [matrix of math nodes,row sep=2em,column sep=4em,minimum width=2em]
    {
     X & X\\
     A & A  \\};
    \path[->]
    (m-1-1) edge node [left] {$f$} (m-2-1)
           edge node [above] {$\alpha_{1_A}$} (m-1-2)
    (m-1-2) edge node [right] {$f$} (m-2-2)
    (m-2-1) edge node [below] {$\alpha_{f}$}  (m-2-2);
  \end{tikzpicture}\]
commutes. (In the terminology of universal algebra, $\alpha_f$\/ \emph{extends} $\alpha_{1_A}$.) 
Moreover, the naturality of $\alpha_f$\/ then amounts to requiring that for each $g:A \to B$, the square
    \[\begin{tikzpicture}
    \matrix (m) [matrix of math nodes,row sep=2em,column sep=4em,minimum width=2em]
    {
     A & A\\
     B & B  \\};
    \path[->]
    (m-1-1) edge node [left] {$g$} (m-2-1)
           edge node [above] {$\alpha_f$} (m-1-2)
    (m-1-2) edge node [right] {$g$} (m-2-2)
    (m-2-1) edge node [below] {$\alpha_{gf}$}  (m-2-2);
  \end{tikzpicture}\]
commutes for any $g\colon A\to B$. 

A morphism $x\colon X\to Y$ induces a homomorphism $\ZZ(X)\to \ZZ(Y)$ as follows: 
first, note that that $x$ induces a functor $x^*\colon Y/\CC\to X/\CC$ fitting into a strictly commuting triangle
    \[\begin{tikzpicture}
    \node (x) at (0,1) {$Y/\CC$};
    \node (a) at (-1,-1) {$X/\CC$};
    \node (b) at (1,-1) {$\CC$};
    \draw[->] (x) to node[left] {$x^*$} (a);
    \draw[->] (x) to node[right] {$P_Y$} (b);
    \draw[->] (a) to node[below] {$P_X$} (b);
    \end{tikzpicture}\]
so that given $\alpha\in\ZZ(X)$ we can define $\ZZ(x)\alpha$ by whiskering along $x^*$, \ie $\ZZ(x)\alpha:=\alpha x^*$. 
In concrete terms, $\ZZ(x)\alpha$ is defined for $f\colon Y\to A$ by $(\ZZ(x)\alpha)=\alpha_{fx}$.  Consequently, $\ZZ$ is functorial in $X$. Moreover, there is a canonical comparison map  $\delta\colon\ZZ(X)\to \Aut(X)$ defined by $\alpha\mapsto\alpha_{\id}$.

\begin{example}  We list some known characterizations of isotropy groups of various categories. Proofs of these results can be found in~\cite{bergman:textbook,parker:thesis,hofstrakarvonen:2d-isotropy,hofstraetal:isotropyofalgtheories}
\begin{itemize}
\item The result mentioned in the introduction states that for a group $G$, the isotropy group $\ZZ(G)$ is naturally isomorphic to $G$.
\item For a monoid, $M$, the isotropy group $\ZZ(M)$ is naturally isomorphic to the group of units of $M$. 
\item For the category of Abelian groups, the isotropy group functor $\ZZ\colon \cat{Ab}\to \Grp$ is isomorphic to the constant functor at $\mathbb{Z}_2$, corresponding to the automorphisms $x\mapsto x$ and $x\mapsto -x$. 
\item For $1$-categories of  categories and groupoids the isotropy group is trivial. In contrast, for the $1$-category of strict monoidal categories and strict monoidal functors, the isotropy group of a monoidal category coincides with the group of units of the underlying monoid, \ie with the group of strictly invertible objects.\footnote{As an aside, one might have hoped that this would change if thinks of categories and groupoids as $2$-categories, and studies ``inner autoequivalences'' by moving to a two-dimensional version of isotropy. However, two-dimensional isotropy remains trivial for categories and groupoids in this setting, whereas for monoidal categories one recovers the $2$-group of weakly invertible objects~\cite{hofstrakarvonen:2d-isotropy}}.
\end{itemize}
\end{example}

\begin{definition} A category $\CC$ has \emph{small isotropy} if for each object $X$ of $\CC$, the class of natural automorphisms of the projection $X/\CC\to\CC$ is in fact a set, so that $\ZZ$ defines a functor $\CC\to\Grp$.
\end{definition}

As shown in~\cite[Theorem 2.4]{hofstrakarvonen:2d-isotropy}, a convenient sufficient condition for a category $\CC$ to have small isotropy is given by $\CC$ having binary coproducts and a small dense subcategory. In particular, any locally presentable category has small isotropy. Whenever $\CC$ has small isotropy, the maps  $\delta_A\colon\ZZ(A)\to \Aut(A)$ defined by $\alpha\mapsto \alpha_{\id}$ promote $\ZZ$ into a crossed $\CC$-module.

For an arbitrary crossed module $(\GG \colon \CC\to\Grp,\gamma)$, the maps $\gamma$ induce maps $\GG(A)\to\ZZ(A)$, given $\alpha\in \GG(A)$, the corresponding element of $\ZZ(A)$ is defined at $f\colon A\to B$ by $\gamma_B(\GG(f) \alpha)$. This endows $\ZZ$ with a universal property.

\begin{proposition}\label{prop:isotropyisterminal}[Essentially proposition 4.11 of \cite{funketal:crossedtoposes}] If $\CC$ has small isotropy, then $\ZZ$ is terminal in the fiber over $\CC$ of the forgetful functor sending a crossed modules to its underlying category. \end{proposition}

We know define our main object of study --- the $2$-category induced by a crossed module.

\begin{definition} The (strict) $2$-category $\CC_\GG$ induced by a crossed $\CC$-module $(\GG \colon \CC\to\Grp,\gamma)$ has $\CC$ as its underlying $1$-category. Given two parallel morphisms $f,g\colon A\rightrightarrows B$ in $\CC$, a $2$-cell $f\to g$ in $\CC_\GG$ is given by an element $\alpha\in \GG(B)$ such that $\gamma_B(\alpha)\circ f=g$. Vertical composition of $2$-cells is given by multiplication in $\GG$. Given $A\xrightarrow{f_1}B\xrightarrow{f_2}C$, $\alpha\colon f_1\to g_1$ and $\beta\colon f_2\to g_2$, the horizontal composite $\beta*\alpha\colon f_2f_1\to g_2 g_1$ is defined as $\beta \cdot \GG(f_2)(\alpha)(= (\GG(g_2)\alpha)\cdot\beta)$.
\end{definition} 

\begin{example} 
\begin{itemize}
\item The $2$-category $\Grp_\ZZ$ gives the aforementioned $2$-category of groups where $2$-cells $\psi\to \phi$ correspond to elements $g$ of the codomain such that $g\psi(-)g^{-1}=\phi$.
\item For monoids, the $2$-category $\cat{Mon}_\ZZ$ gives a similar $2$-category, where $2$-cells $\psi\to \phi$ correspond to invertible elements $g$ of the codomain such that $g\psi(-)g^{-1}=\phi$.
\item For the category of Abelian groups, the nontrivial $2$-cells in $\cat{Ab}_\ZZ$ are of the form $(-1)\colon f\to -f$ for any homomorphism $f$. 
\item If the isotropy group of a category $\CC$ is trivial, then $\CC_\ZZ$ is just $\CC$ viewed as a locally discrete $2$-category. Given Proposition~\ref{prop:isotropyisterminal}, in this case for any crossed module $(\GG \colon \CC\to\Grp,\gamma)$ on $\CC$, all $2$-cells in $\CC_\GG$ are automorphisms.
\end{itemize}
\end{example}

One could carry out straightforward but tedious calculations to show that this indeed results in a strict $2$-category. Instead, we rely on~\cite[Section 5]{funketal:crossedtoposes} where this is deduced more abstractly: in a nutshell, just like an ordinary crossed module of groups induces (and indeed is equivalent to) a group object in $\cat{Cat}$, a crossed module in the sense of Definition~\ref{def:crossedmod} can be shown to induce a category object in $\cat{Cat}$, \ie  a double category. This double category turns out to be one coming from a $2$-category with further special properties.

\begin{definition}[Dual of Definition 5.17 of~\cite{funketal:crossedtoposes}] A $2$-category is \emph{right-generated by contractible loops} if every $\alpha\colon f\to g$ can be written uniquely as $\alpha=\beta f$, where $\beta$ has the identity as its domain, \ie if every $\alpha$ satisfies 
    \[\begin{tikzpicture}[baseline=-.15cm]
    \node (a) at (-1,0) {$A$};
    \node (b) at (1,0) {$B$};
    \draw[->] (a) to[out=35,in=145] coordinate[midway] (h1) node[above]{$f$} (b);
    \draw[->] (a) to[out=-35,in=215] coordinate[midway] (h2) node[below] {$g$} (b);
    \draw[shorttwocell] (h1) to node[right] {$\alpha$} (h2);
    \end{tikzpicture}
    \quad=\quad
    \begin{tikzpicture}[baseline=-.15cm]
    \node (c) at (-3,0) {$A$};
    \node (a) at (-1,0) {$B$};
    \node (b) at (1,0) {$B$};
    \draw[->] (c) to node[above] {$f$}  (a);
    \draw[->] (a) to[out=35,in=145] coordinate[midway] (h1) node[above]{$\id$} (b);
    \draw[->] (a) to[out=-35,in=215] coordinate[midway] (h2) node[below] {$\cod \beta$} (b);
    \draw[shorttwocell] (h1) to node[right] {$\beta$} (h2);
    \end{tikzpicture}
    \]
  for a unique $\beta$.
\end{definition}

\begin{theorem}\label{thm:xmodto2cat}[Theorem 5.18 of \cite{funketal:crossedtoposes}] The assignment $(\GG \colon \CC\to\Grp,\gamma)\mapsto \CC_\GG$ extends to a (strict) $2$-functor from crossed modules to $2$-categories that is full and faithful on both 1 and $2$-cells. Its essential image is given by $2$-categories that are locally groupoidal and right-generated by contractible loops.
\end{theorem}

\begin{remark} Another viewpoint on inner automorphisms of a group is that an automorphism $f\colon G\to G$ is inner iff it is isomorphic to $\id[G]$ in the $2$-category of groups. This works in fact more generally: if one thinks of a crossed $\CC$-module structure $(\GG,\gamma)$ as giving an abstract group of structured inner automorphisms, then 
 one could define an automorphism $f\colon A\to A$ in $\CC$ to be inner iff it is isomorphic to $\id[A]$ in $\CC_\GG$.
\end{remark}

\begin{remark} As any category $\CC$ with small isotropy has a canonical crossed module given by the isotropy group, one could envision first promoting $\CC$ into a $2$-category $\CC_\ZZ$ and then applying this construction to each hom-category ad infinitum. However, this construction is not worth iterating: for $n>2$ one can show that in the resulting $(n,1)$-category there is always exactly one $n$-cell between two parallel $n-1$-cells.
\end{remark} 

\section{Two-dimensional (co)limits in $2$-categories induced by crossed modules}

In this section we study the existence of two-dimensional (co)limits in $2$-categories induced by crossed modules. We begin with the only positive result we know to hold without further assumptions. 

\begin{theorem}\label{thm:connectedcolims}  Let $(\GG \colon \CC\to\Grp,\gamma)$ be a crossed module. Then all connected colimits that exist in $\CC$ are (strict) $2$-colimits in $\CC_\GG$. 
\end{theorem}

\begin{proof}
 Let $\JJ$ be a connected category and consider the ($1$-categorical) colimit $(C,(c_j)_{j\in\JJ})$ of  $D\colon \JJ\to \CC$. To show that it has the required two-dimensional universal property, consider two cocones $(b_j)_{j\in\JJ}$ and $(b_j')_{j\in\JJ}$ on an object $B$ and a modification $(\alpha_j\colon b_j\to b_j')_{j\in \JJ}$ between them. The two-dimensional universal property states that any such modification is induced by a unique $2$-cell $\alpha\colon b\to b'$, where $b,b'\colon C\rightrightarrows B$ are the canonical maps factorizing the cocones  $(b_j)_{j\in\JJ}$ and $(b_j')_{j\in\JJ}$ respectively. To show that this holds, note first that for any $f\colon j\to k$ in $\JJ$, we have an equality of $2$-cells $\alpha_{k} b_{k}D(f)=\alpha_j$ as $(\alpha_j)_{j\in \JJ}$ is a modification.
 As two parallel $2$-cells are equal iff the corresponding elements of $\GG$ are, this means that $\alpha_{k}=\alpha_{j}$ whenever there is a map between $j$ and $k$ in $\JJ$. As $\JJ$ is connected, this implies that $\alpha_{j}=\alpha_{k}$ for all $j,k\in\JJ$, so let us write $\alpha$ for the element of $\GG(B)$ that is uniquely determined by this modification. As the cocone on $C$ is jointly epic, the equations $\gamma(\alpha) b_j=b'_j$ imply that $\gamma(\alpha) b=b'$, so that $\alpha$ determines a $2$-cell $b\to b'$. Thus the modification $(\alpha_j\colon b_j\to b_j')$ is induced by a unique $2$-cell $\alpha\colon b\to b'$ as desired.
\end{proof}

The following simple fact will be used repeatedly.
\begin{lemma}\label{lem:initial} Let  $(\GG \colon \CC\to\Grp,\gamma)$ be a crossed module and let $\II\colon \Grp\to\Cat$ be the inclusion. If $\CC$ has an initial object $0$, then there is a $2$-natural isomorphism \[\CC_\GG(0,-)\cong \II\circ\GG .\]
\end{lemma}

\begin{proof} As $0$ is initial in $\CC$, for any object $A$ the category $\CC_\GG(0,A)$ has only one object corresponding the unique $1$-cell $0\to A$, so that morphisms in $\CC_\GG(0,A)$ correspond to its automorphisms. In turn, such automorphisms correspond to elements of $\GG(A)$ as the equation required from $2$-cells of $\CC_\GG$ holds vacuously for $0$.
\end{proof}

Many of our results hold under the assumption that the presheaf of sets underlying the presheaf of groups is representable. We now give an equivalent characterization of when this happens. 

\begin{theorem}\label{thm:representability} Let $(\GG \colon \CC\to\Grp,\gamma)$ be a crossed module and assume that $\CC$ has an initial object. Then the following are equivalent: 
\begin{enumerate}[label=(\roman*)]
  \item  The tensor $\mathbb{Z}\otimes 0$ of the initial object $0$ of $\CC$ by $\mathbb{Z}\in\Cat$ exists in $\CC_\GG$.
  \item The composite $\UU\circ\GG\colon\CC\to\Set$ is representable, where $\UU\colon \Grp\to\Set$ is the forgetful functor.
\end{enumerate}
Moreover, if $\CC$ is also cocomplete, then these are equivalent to $\GG$ having a left adjoint.
\end{theorem}
\begin{proof}
The tensor $\mathbb{Z}\otimes 0$ is defined by the isomorphism $\CC_\GG(\mathbb{Z}\otimes 0,-)\cong\Cat(\mathbb{Z},\CC_\GG(0,-))$, which in turn is isomorphic to $\Cat(\mathbb{Z},\II\circ\GG(-))$ by Lemma~\ref{lem:initial}. As these isomorphism are strictly $2$-natural, we can forget the $2$-cells and obtain a $1$-natural transformation, which we can then whisker with the $1$-functor $\Ob\colon\Cat\to\Set$ sending a small category to its set of objects. Moreover, the objects of $\Cat(\mathbb{Z},\II\circ\GG(A))$ correspond to group homomorphisms $\mathbb{Z}\to\GG(A)$ which in turn correspond to elements of $\UU\GG(A)$. As a result, we have a natural isomorphism
\[\CC(\mathbb{Z}\otimes 0,-)\cong \Ob(\CC_\GG(\mathbb{Z}\otimes 0,-))\cong \Ob(\Cat(\mathbb{Z},\II\circ\GG(-)))\cong \UU\circ\GG.\]
This means that (i) implies (ii) as the one-dimensional universal property of $\mathbb{Z}\otimes 0$ is exactly the same as the universal property of an object representing $\UU\circ\GG$.  

Let us us now analyze the two-dimensional universal property by considering the morphisms of $\Cat(\mathbb{Z},\GG(A))$. Given two objects of $\Cat(\mathbb{Z},\GG(A))$ corresponding to elements $\alpha,\beta\in\GG(A)$, a morphism $\alpha\to \beta$ in $\Cat(\mathbb{Z},\GG(A))$ is given by an element $\theta\in\GG(A)$ satisfying $\theta \alpha=\beta\theta$. The two-dimensional universal property of $\mathbb{Z}\otimes 0$ asserts that such a $\theta$ is induced by a unique $2$-cell between the corresponding $1$-cells $f_\alpha,f_\beta\colon \mathbb{Z}\otimes 0\rightrightarrows A$ in in $\CC_\GG$. 

With this in mind, let us now show that (ii) implies (i), so let $R$ be an object of $\CC$ representing $\UU\circ\GG$. Let us fix a universal element $\tau\in \UU\circ\GG(R)$ so that for every $A$ and $\alpha\in\GG(A)$ there exists a unique $f_\alpha\colon R\to A$ in \cat{C} such that $\GG(f_\alpha)\tau=\alpha$. We now need to show that, given two maps $f_\alpha,f_\beta\colon R\rightrightarrows A$ corresponding to elements $\alpha,\beta\in \GG(A)$ and a $2$-cell $\theta\colon f_\alpha\circ !_R\to f_\beta\circ !_R$ satisfying $\theta \alpha=\beta\theta$
, then $\theta$ already defines a $2$-cell $f_\alpha\to f_\beta$. In other words, we wish to show that $\gamma_A(\theta)\circ f_\alpha=f_\beta$. As maps out of $R$ are defined uniquely by where they map $\tau$, to check this we can equivalently compare where the two sides send $\tau$. On the right hand side, we get $\GG(f_\beta)\tau=\beta$. On the left hand side, we get 
\[\GG(\gamma_A(\theta)) (\GG(f_\alpha)\tau)=\GG(\gamma_A(\theta)) \alpha=\theta \alpha \theta^{-1}=\beta,\]
where the penultimate equality is the Peiffer identity identity and the last condition is equivalent to our condition on $\theta$. Thus (ii) indeed implies (i), as any object satisfying the one-dimensional universal property of $\mathbb{Z}\otimes 0$ automatically satisfies the two-dimensional one.

Assume now that $\CC$ is cocomplete. If $\GG$ has a left adjoint $L$, then $L(\mathbb{Z})$ is an object representing $\UU\circ \GG$. Conversely, if $\UU\circ\GG$ is representable, the existence of a left adjoint follows from~\cite[Theorem 10.4.3]{bergman:textbook}, see also~\cite{freyd:algebravaluedfunctors}.
\end{proof}

\begin{example} 
\begin{itemize}
  \item The isotropy functor $\ZZ\colon\cat{Mon}\to\Grp$ sending a monoid to its invertible elements has a left adjoint given by the inclusion $\Grp\hookrightarrow \cat{Mon}$. 
  \item The techniques of~\cite{hofstraetal:isotropyofalgtheories,parker:thesis} can be used to show that the isotropy group of a semiring is given by its group of units. This has a left adjoint given by sending a group to its group semiring. 
  \item There is a similar adjunction sending a (unital but not necessarily commutative) ring to its group of units and sending a group to its group ring. More generally, one can send an algebra over a ring to its group of units 
  with the left adjoint given by sending a group to its group algebra. However, these adjunctions are not of the form $L\dashv \ZZ$, because $\ZZ_{\cat{R-alg}}$ is not given by sending a ring to its group of units. To see why, note that if $u$ is a unit in a ring, then conjugating with $u$ and with $-u$ gives rise to the same (extended) inner automorphism. In fact, Bergman~\cite{bergman:inner} shows that the isotropy group of an algebra over a field is given by the group of units of the algebra modulo the group of units of the field. 
\end{itemize}
\end{example}

From now on, we say that a crossed module $(\GG \colon \CC\to\Grp,\gamma)$ is representable whenever the composite  $\UU\circ\GG\colon\CC\to\Set$ is representable. Note that if $\CC$ has finite coproducts, the representing object is equipped with the structure of a cogroup object, \ie a group object in $\CC\op$.
We now give our main positive results concerning two-dimensional (co)limits in s$\CC_\GG$.  

\begin{theorem}\label{thm:positiveresults} Let $(\GG \colon \CC\to\Grp,\gamma)$ be a representable crossed module.
Then
  \begin{enumerate}[label=(\roman*)]
    \item All limits in $\CC$ are $2$-limits in  $\CC_\GG$.
    \item If $\CC$ is finitely cocomplete, then $\CC_\GG$  has coinserters, coequifiers and coidentifiers.
    \item If $\CC$ is cocomplete, then $\CC_\GG$ has all tensors by a monoid.
  \end{enumerate}
\end{theorem}

\begin{proof}
(i) Given a diagram $D\colon \cat{J}\to \CC$ and a limit $(L,(l_j)_{j\in \cat{J}})$ of $D$ in $\CC$, we need to show that $L$ satisfies the required two-dimensional universal property in $\CC_\GG$. That is, given two maps $h,k\colon A\rightrightarrows L$ and $2$-cells $\alpha_j\colon l_j h\to l_j k$ satisfying 
  \begin{equation}\tag{$*$}\label{eq:modificationoflimitcones}D(f)\alpha_j=\alpha_k \text{ for every }f\colon j\to k\text{ in }\cat{J}\end{equation}
we need to show that there is a unique $2$-cell $\alpha\colon h\to k$ such that $\alpha_j=l_j\alpha$ for each $j\in\cat{J}$. Unwinding the definition of $\CC_\GG$, we have elements $\alpha_j\in \GG(D(j))$ such that $\gamma(\alpha_j)l_jh=l_jk$ for every $j\in\cat{J}$ and $\GG D(f)\alpha_j=\alpha_k$ for every $f\colon j\to k$ in $\cat{J}$. As $\UU\circ\GG$ is representable, $\GG$ is continuous and hence there exists a unique $\alpha\in\GG(L)$ such that $\GG(l_j)\alpha=\alpha_j$ for each $j$. It remains to check that $\alpha$ defines a $2$-cell $h\to k$, \ie that $\gamma(\alpha) h=k$. But this follows since, for any $j$ we have 
\[l_j\gamma(\alpha) h=\gamma(\alpha_j) l_jh=l_jk\]
by \eqref{eq:modificationoflimitcones}. 
 
(ii) Let $(R,\tau\in \GG(R))$ be a representation of $\UU\circ \GG$. To construct the coinserter of $f,g\colon A\rightrightarrows B$, we mimic the construction in $\Grp$. We first form the coproduct $B+R$, which intuitively speaking adds an inner automorphism to $B$ freely. We then take a quotient to force this inner automorphism to become a $2$-cell between the projections of $f$ and $g$ to the quotient. More formally, let $c\colon B+R\to C$ be the coequalizer of $\gamma(\GG(i_R)\tau)i_B f, i_B g\colon A\rightrightarrows B+R$. Denote $\GG(ci_R)\tau\in \GG(C)$ by $\alpha$. Now $\alpha$ defines a $2$-cell $ci_Bf\to ci_Bg$ since 
\[\gamma(\alpha) ci_Bf=\gamma(\GG(ci_R)\tau)ci_Bf=c\gamma(\GG(i_R)\tau)i_Bf=ci_Bg.\]
We claim that $(C,\alpha)$ is the coinserter of $f$ and $g$. To check the one-dimensional universal property, consider an arbitrary $k\colon B\to D$ with $\beta\in \GG(D)$ defining a $2$-cell $kf\to kg$. The one-dimensional universal property states that any such $k$ factors uniquely through the canonical map $c\circ i_B B\to C$, say via $q\colon C\to D$, and that the $2$-cell $\beta\colon kf\to kg$ satisfies $\beta=q\alpha$. 

To prove this, let $h\colon R\to D$ be the unique map such that $\GG(h)\tau=\beta$. Then $[k,h]\colon B+R$ satisfies 
\[[k,h]\gamma(\GG(i_R)\tau) i_B f=\gamma(\GG(h)\tau)[k,h] i_B f=\gamma(\beta) k f=kg=[k,h]i_Bg\]
and hence factors uniquely through $c$ via $q\colon C\to X$. By construction, $\GG(q)\alpha=\beta$ and $k=q ci_B$, and clearly $q$ is uniquely determined by these equations. 

To show that $(C,\alpha)$ has the appropriate two-dimensional universal property, consider arbitrary maps $q,r\colon C\to D$ and $\beta\in \GG(D)$ defining a $2$-cell $qci_B\to rci_B$ such that 
\begin{equation}\tag{$*$}\label{eq:pastingassumption}(\beta g)(q \alpha)=(r\alpha)(\beta g)\end{equation}
The two-dimensional universal property states that any such $2$-cell $\beta$ is induced by a unique $2$-cell $q\to r$. Thus it suffices to show that $\beta\in \GG(D)$ defines a $2$-cell $q\to r$, \ie that $\gamma(\beta) q=r$. As $c\colon B+R\to C$ is epic, it suffices to show that $\gamma(\beta) qc=rc$, which in turn can be checked by precomposing with $i_B$ and $i_R$. For $i_B$ this follows from the assumption that $\beta$ defines a $2$-cell $qci_B\to rci_B$. For $i_R$, we wish to show that
 \begin{equation}\tag{$**$}\label{eq:2dUMP}\beta_{\id} qci_R=rci_R.\end{equation}
 
 As maps out of $R$ are determined by where they send $\tau\in \GG(R)$, let us compute it for both sides of \eqref{eq:2dUMP}. Recall that $\alpha=\GG(ci_R)\tau$, so that left side sends $\tau$ to  $\GG(\gamma(\beta))\GG(q)\alpha=\beta (\GG(q)\alpha)\beta^{-1}$ by the Peiffer identity. The right hand side sends $\tau$ to $\GG(r)\alpha$. Thus~\eqref{eq:2dUMP} is equivalent to 
 \[\beta (\GG(q)\alpha)=(\GG(r)\alpha)\beta\] 
 which in turn follows from~\eqref{eq:pastingassumption}, as desired.

To find the coequifier of $\alpha,\beta\colon f\rightrightarrows g$ where $f,g\colon A\rightrightarrows B$, consider the two maps from the representing object $R$ that correspond to $\alpha$ and $\beta$, and take their coequalizer $c\colon B\to C$. By construction, $c\alpha=c\beta$, so let us show that $C$ is universal as such. Indeed, if $d\colon B\to D$ satisfies $d\alpha=d\beta$, then $d$ factors uniquely through the coequalizer. To show the two-dimensional universal property, consider $h,k\colon C\rightrightarrows D$ and an arbitrary $2$-cell $\theta\colon hc\to kc$. The two-dimensional universal property states that $\theta$ is induced by a unique $2$-cell $h\to k$. Now,   $\theta\in \GG(D)$ satisfies $\gamma(\theta) hc=kc$, so that $\gamma(\theta) h=k$ as $c$ is epic. Thus $\theta$ defines a $2$-cell $h\to k$ as desired, and is clearly unique as such.

The coidentifier of $\alpha\colon f\to g$ where $f,g\colon A\rightrightarrows B$ is constructed by taking the coequalizer of the two maps $R\to D$ that correspond to $\alpha,0\in\ZZ(B)$. The proof of the universal property is similar to the case of a coequifier.

For (iii), we will first construct the tensor $G\otimes B$ when $G$ is a free group on some set $X$. First, let $(R,\tau)$ represent $\UU\circ \GG$, and form the coproduct $B+\coprod_{x\in X} R$. Write $\tau_x$ for the $x$-th inclusion of $\tau$ (\ie the element $\GG(i_x)\tau\in \GG(B+\coprod_{x\in X} R)$), and let $c\colon B+\coprod_{x\in X} R\to  G\otimes B$ denote the (wide) coequalizer of $i_B$ with $\gamma(\tau_x) i_B$ for every $x\in X$. Writing $\alpha_x$ for the image of $\tau_x$ under the projection to $C$, we have $\gamma(\alpha_x)_c i_B=c\gamma(\tau_x) i_B=ci_B$, so each $\alpha_x$ defines an endo-$2$-cell on $ci_B\colon B\to G\otimes B$, and hence we have a functor $G\to \CC_\GG(B,G\otimes B)$. Let us now show that it is universal as such. For the one-dimensional universal property, we need to show that for any object $C$ of $\CC$, any functor $G\to \CC_\GG(B,C)$ factors via the canonical one $G\to \CC_\GG(B,G\otimes B)$ through a unique $1$-cell $g\colon G\otimes B\to C$. As $G$ is free, any such functor is given by a choice of a morphism $k\colon B\to C$ and $2$-cells $\beta_x\colon k\to k$ for $x\in X$. Now, by the universal property of  $B+\coprod_{x\in X} R$  there is a unique map $f\colon B+\coprod_{x\in X} R$ such that $fi_B=k$ and $\GG(f)\tau_x=\beta_x$. Now 
\begin{equation}\tag{$*$}\label{eq:2d-umpfortensor}f\gamma(\tau_x)i_B=\gamma(\beta_x) fi_B=\gamma(\beta_x)k=k=fi_B \end{equation}
so that $f$ factors uniquely via $g\colon G\otimes B\to C$. Now, $g$ satisfies $gci_B=k$ and $g\alpha_x=\beta_x$, and clearly $g$ is unique as such. 

For the two-dimensional universal property, consider two maps $f,g\colon G\otimes B\to C$ and a $2$-cell $\beta\colon fci_B\to gci_B$ satisfying
\[(\beta)(f\alpha_x)=(g\alpha_x)\beta.\]
The two-dimensional universal property states that any such $2$-cell is induced by a unique $2$-cell $f\to g$. The only possible choice for such a $2$-cell is given by  $\beta\in\GG(C)$, so it is enough to show that $\beta$ defines a $2$-cell $f\to g$. As $c$ is epic it suffices to show that $\gamma(\beta) fc=\gamma(\beta) gc$. It remains to show that this equation is true when precomposed by any of the summands of $B+\coprod_{x\in X} R$. For $i_B$ this follows from the assumption that $\beta$ defines a $2$-cell $fci_B\to gci_B$. For $x\in X$, it suffices to compute the image of $\alpha_x$: 
\[\GG(\gamma(\beta))\GG(f)\alpha_x=\beta(\GG(f)\alpha_x)\beta^{-1}\]
where the first equation uses the Peiffer identity  and the second comes from \eqref{eq:2d-umpfortensor}.

Now, an arbitrary group $G$ can be given as a coequalizer, and in particular, as a connected colimit (in $\Grp$) of free groups. Moreover, the inclusion $\Grp\to\Cat$ preserves connected colimits.
As $(-)\otimes B$ is cocontinuous in the group $G$, the tensor $G\otimes B$ hence exists due to $\CC_\GG$ having connected colimits by Theorem~\ref{thm:connectedcolims}.

Finally, as any $2$-cell in $\CC_\GG$ is invertible, for any monoid $M$ any functor $M\to \CC_\GG(A,B)$ factors uniquely via the group $G$ obtained by universally adding inverses to elements of $M$, so that $M\otimes B$ is isomorphic to $G\otimes B$. 
\end{proof}

For (i), note that the one-dimensional universal property is known to imply the two-dimensional one if the $2$-category in question has cotensors by the walking arrow category $\cat{2}$. However, we will see that this particular limit cannot exist unless $\GG$ is trivial, so our result does not follow from this standard fact. For (ii), note that our construction of the coinserter mimics the usual construction of HNN-extensions in $\Grp$. Curiously, the construction uses coproducts in $\CC$ to construct a strict $2$-colimit in $\CC_\GG$, despite the fact that coproducts in $\CC$ rarely satisfy the two-dimensional universal property in $\CC_\GG$, as we will see. 

For (iii), note that the existence of tensors is known to follow from coproducts, coinserters and coequifiers~\cite[Dual of Proposition 4.4.]{kelly:2-dlimits}. However, the coproducts in $\CC$ do not in general satisfy the two-dimensional universal property in $\CC_\GG$, so another proof strategy is required. In fact, $2$-categorical coproducts (and many other (co)limits) cannot exist in $\CC_\GG$ unless $\GG$ is trivial.

\begin{theorem}\label{thm:obstructionsforcolims}[Obstructions for colimits]  Let $(\GG \colon \CC\to\Grp,\gamma)$ be a crossed module  where $\CC$ has an initial object $0$. Then the following are equivalent:
  \begin{enumerate}[label=(\roman*)]
    \item $\GG$ is trivial, \ie each $\GG(A)$ is the trivial group.
    \item The coproduct $0+0\cong 0$ is a $2$-coproduct in $\CC_\GG$.
    \item The initial object of $\cat{C}$ is $2$-initial in $\CC_\GG$.
    \item $\cat{C}_\GG$ has tensors by some fixed category $\DD$ that is not a monoid.
  \end{enumerate}
\end{theorem}

\begin{proof}

If $\GG(A)=0$ for all $A$ in $\CC$, then $\CC_\GG$ is locally discrete so (ii)-(iv) hold trivially. 

``(ii)$\Rightarrow $(i)'':
If $0+0\cong 0$ is a two-dimensional coproduct, then  Lemma~\ref{lem:initial} implies that 
  \[\II\circ\GG(-)\cong \CC_\GG(0,-)\cong \CC_\GG(0+0,-)\cong \CC_\GG(0,-)\times \CC_\GG(0,-)\cong \II\circ \GG(-)\times \II\circ \GG(-).\] Moreover, this isomorphism is the diagonal map as it is induced by the isomorphism $0\cong 0+0$, forcing $\GG$ to be trivial.

``(iii)$\Rightarrow $(i)'':
If $0$ is $2$-initial, then by Lemma~\ref{lem:initial} we have $\II\circ\GG(-)\cong \CC_\GG(0,-)\cong\Delta\cat{1}$, proving the claim. 

``(iv)$\Rightarrow $(i)'': the tensor of an object $A$ with the empty category $\cat{0}$, if it were to exist, would be an object $\cat{0}\otimes A$ such that $\CC_\GG(\cat{0}\otimes A,-)\cong\Cat(\cat{0},\CC_\GG(A,-))\cong\Delta\cat{1}$, \ie a $2$-initial object, so the claim follows from (iii)$\Rightarrow $(i).

Let $\DD$ be a small category with more than one object and assume that the tensor $\DD\otimes A$ of $A$ with $\DD$ exists, and consider the constant functor $\Delta(\id)\colon\DD\to\CC_\GG(A,A)$ at $\id[A]$. Pick any two $\alpha,\beta\in\GG(A)$, and define for each object of $\DD$ a $2$-cell out of $\id[A]$ by picking $\alpha\colon \id[A]\to \gamma(\alpha),\beta\colon  \id[A]\to \gamma(\beta)$ for two distinct objects and the identity for everything else: there is a unique functor $F\colon\DD\to\CC_\GG(A,A)$ for which this defines a natural transformation $\Delta(\id)\to F$. By the two-dimensional universal property of $\DD\otimes A$, this natural transformation has to be induced by a unique $2$-cell between the maps $\DD\otimes A \to A$ corresponding to $\Delta(\id)$ and $F$, so that in particular we must have $\alpha=\beta$, so that $\GG(A)=0$. 
\end{proof}

The implication ``(ii)$\Rightarrow $(i)'' above can be readily strengthened to show more specifically that $\GG(C)=0$ whenever $A+B$ is a two-dimensional coproduct in $\CC_\GG$ and there are maps $A\rightarrow C\leftarrow B$: just consider arbitrary two elements of $\GG(C)$ and let one act on the map from $A$ and the other on the map from $B$. Consequently, a two-dimensional coproduct can only exist in $\CC_\GG$ in the trivial case when its two-dimensional universal property holds trivially. Note also that the obstructions above really come from the two-dimensional properties, and hence these hold even if one worked with bicategorical colimit notions replacing $[\JJ,\Cat]$ with the category of pseudofunctors, pseudonatural transformations and modifications, and required that the (co)limit represents the appropriate functor only up to pseudonatural equivalence rather than up to $2$-natural isomorphism.

While all limits of $\CC$ become $2$-limits in $\CC_\GG$ when $\GG$ is representable, we now show that there are hardly any genuinely $2$-categorical limits in $\CC_\GG$.

\begin{theorem}\label{thm:limshavetrivial2cells}[A general obstructions for limits] Let  $(\GG \colon \CC\to\Grp,\gamma)$ be a crossed module where $\CC$ has an initial object. Let
 $\JJ$  be a locally discrete $2$-category, $W\colon \JJ\to \Cat$ a weight and $D\colon \JJ\to\CC$ be a diagram in $\CC$. If the weighted limit $\lim^W D$ exists in $\CC_\GG$, then the $2$-cells appearing in every $W$-weighted cone for $D$ are trivial. 

In this case, the weighted two-dimensional limit $\lim^W D$ in $\CC_\GG$ can be equivalently described as the limit $\lim^{\Pi_0 W} D$  in $\CC$, where $\Pi_0\colon \Cat\to\Set$ is the  functor sending a category to its connected components. Hence, all existing weighted limits in $\CC_\GG$ where the indexing shape is a mere category can be expressed as conical limits in $\CC$.
\end{theorem}

\begin{proof} 

Let us build a trivial cone on $0$ by defining a natural transformation $W\to \CC_\GG(0,D(-))$ whose $j$-th component for $j\in\JJ$ is given by the constant functor $W(j)\to \CC_\GG(0,D(j))$ on the unique object of $\CC_\GG(0,D(j))$. As $\JJ$ has no nontrivial $2$-cells, this is a $2$-natural transformation and hence defines a trivial cone on $0$. If $\lim^W D$ exists, this trivial cone must factor via the universal one by a map $0\to \lim^W D$. As two $2$-cells between maps into an object $A$ are equal iff the corresponding elements of $\GG(A)$ are equal, the fact that the trivial cone factors via the universal one implies that all the $2$-cells appearing in the universal cone and hence in every cone must also be trivial. 

Consequently, for every object $A$, every natural transformation $W\to \CC_\GG(A,D(-))$ factors via $\Delta\circ \CC(A,D(-))$ where $\Delta\colon \Set\to\Cat$ promotes a set into a discrete category. As $\Delta$ is the right adjoint to $\Pi_0$, $2$-natural transformations $W\to \CC_\GG(A,D(-))$ thus correspond to ordinary natural transformations  $\Pi_0\circ W\to \CC(A,D(-))$ which are $\Set$-valued functors. Consequently, the universal property satisfied by $\lim^W D$ is that of $\lim^{\Pi_0 W} D$, which in turn can be expressed as a conical limit as those are sufficient for ordinary (\ie $\Set$-enriched) categories.
\end{proof}

\begin{corollary}\label{cor:obstructionslims}[Specific obstructions to limits]  Let  $(\GG \colon \CC\to\Grp,\gamma)$ be a crossed module where $\CC$ has an initial object. 
If $\GG(A)$ is nontrivial, then the following limits cannot exist in $\CC_\GG$:
          \begin{enumerate}[label=(\roman*)]
           \item the cotensor of $A$ by $\cat{2}$ 
            \item lax limits of arrows into $A$
            \item  comma objects of maps $X\rightarrow A\leftarrow Y$
            \item inserters of two maps $X\rightrightarrows A$
            \item equifiers of distinct parallel $2$-cells between maps into $A$.
          \end{enumerate}
 \end{corollary}

\begin{proof}

(i)-(iv) follow from Theorem~\ref{thm:limshavetrivial2cells} by using any non-zero element of $\GG(A)$ to exhibit a cone with a nontrivial $2$-cell. For instance, to rule out the comma object of 
\[f\colon X\longrightarrow A\longleftarrow Y\colon g,\] 
pick a non-zero $\alpha\in\GG(A)$ and consider the diagram
      \[\begin{tikzpicture}
    \matrix (m) [matrix of math nodes,row sep=2em,column sep=4em,minimum width=2em]
    {
     0 & Y\\
     X& A \\};
    \path[->]
    (m-1-1) edge node [left] {} coordinate[midway] (a) (m-2-1)
           edge node [above] {} (m-1-2)
    (m-1-2) edge node [right] {$g$} (m-2-2)
    (m-2-1) edge node [below] {$f$} coordinate[midway] (b) (m-2-2);
    \draw[shorttwocell](m-2-1) to node [above,sloped] {$\alpha$} (m-1-2);
  \end{tikzpicture}\]
displaying a cone with a nontrivial $2$-cell.
For (v), we prove the contrapositive.  Consider any two $\alpha,\beta\in \GG(A)$ and the corresponding endo-$2$-cells on $0\to A$  and take their equifier $E\to 0$. As two $2$-cells are equal iff the corresponding elements of $\GG(A)$ are equal, the existence of this equifier implies that $\alpha=\beta$.
\end{proof}

Again, our negative results on limits hold even if one worked with bicategorical limits.

While our prior theory applies generally to all crossed modules, we record one specific result concerning isotropy. We rely on the characterization from~\cite{hofstraetal:isotropyofalgtheories} which we'll recall next up to a suitable precision. Let $\CC$ be the category of models of a single-sorted algebraic theory. Given an object $M$ of $\CC$, we write  $M\adjoin{x}$ for the coproduct of $M$ with the free model on one generator, and more generally $M\adjoin{x_1,\dots x_k}$ for the coproduct of $M$ with the free model on $k$ generators. As the notation suggests, we think of  $M\adjoin{x}$ as the result of adjoining an indeterminate $x$ to $\CC$ freely, so that elements of $M\adjoin{x}$ correspond to terms built from $M$ with one free variable, or rather, to equivalence classes of such terms modulo the theory. Given two terms $t(x),s(x)\in M\adjoin{x}$, the process $(t,s)\mapsto t[s/x]=t(s(x))$ of substituting $s$ for $x$ in $t$ results in a monoid structure on the (underlying set of) $M\adjoin{x}$, which we call the substitution monoid.

The neutral element of the substitution monoid is given by the term $x$ (or rather, its equivalence class modulo the equations of the theory), and a term is invertible if it is invertible in the substitution monoid. A term $t\in M\adjoin{x}$ is said to commute generically with the $k$-ary operation $f$ if the terms $t[f(x_1,\dots x_k)/x]$ and $f(t[x_1/x],\dots t[x_k/x])$ are equal as elements of $M\adjoin{x_1,\dots ,x_k}$.

Now, the main result of \cite{hofstraetal:isotropyofalgtheories} states\footnote{To be precise,\cite{hofstraetal:isotropyofalgtheories} only states its result for the category of finitely presentable models: however, the same characterization holds for all models~\cite[Corollary 2.3.3]{parker:thesis}} that the isotropy group of an object $M$ of an algebraic theory is naturally isomorphic with the group of invertible elements of the substitution monoid $M\adjoin{x}$ that commute generically with all function symbols of the theory.

\begin{theorem} If $\CC$ is the category of models of an algebraic theory with a finite signature, then $\ZZ$ preserves filtered colimits. Thus if $\ZZ$ is representable, the representing object is finitely presentable.
\end{theorem}

\begin{proof}
Recall first  that a $M$ model of an algebraic theory is finitely presentable iff $\hom(M,-)$ preserves filtered colimits, see e.g.~\cite[Proposition 13.26]{adamekrosickyvitale:algebraictheories}. This takes care of the second claim, and will also be used for the first claim. 

We prove the first claim by analyzing three functors: (i) the functor $S\colon \CC\to\cat{Mon}$ sending  a model $M$ to   the substitution monoid $M\adjoin{x}$, (ii) the subfunctor $P$ of it sending a $M$ the elements of  $M\adjoin{x}$ that commute generically with all function symbols of the theory, and (iii) the functor $I\colon \cat{Mon}\to\Grp$ sending a monoid to its group of invertible elements. 

Now, the aforementioned characterization states that $\ZZ_\CC$ is isomorphic to $IP$, so it suffices to prove that this composite preserves filtered colimits. Moreover, $I$ is isomorphic to $\cat{Mon}(\mathbb{Z},-)$ and $\mathbb{Z}$ is a finitely presented monoid, so that $I$ preserves filtered colimits. 

We  now show that $S$ preserves filtered colimits. Indeed, let $D\colon \JJ\to \CC$ be a filtered diagram, and consider the canonical comparison $\colim SD\to S(\colim D)$ map in $\cat{Mon}$. To prove that it is an isomorphism in $\cat{Mon}$, it suffices to prove that the underlying function is an isomorphism in $\Set$. As the forgetful functor $U\colon \cat{Mon}\to\Set$  to set is represented by $\mathbb{N}$ which is finitely presentable, it suffices to prove that the canonical comparison $\colim USD\to US(\colim D)$ is an isomorphism. In turn, as $US(M)\cong M\adjoin{x}$ naturally in $M$ it is sufficient to show that the endofunctor on $\CC$ defined by $M\mapsto\adjoin{x}$ preserves filtered colimits, but this is true as colimits commute with colimits.

We conclude by showing that $P\hookrightarrow S$ preserves filtered colimits. Now, the canonical comparison map $\colim PD\to P\colim D$  fits into a commutative diagram  
  \[\begin{tikzpicture}
    \matrix (m) [matrix of math nodes,row sep=2em,column sep=4em,minimum width=2em]
    {
     \colim PD& P\colim D\\
     \colim SD& S\colim D\\};
    \path[->]
    (m-1-1) edge node [left] {} coordinate[midway] (a) (m-2-1)
           edge node [above] {} (m-1-2)
    (m-1-2) edge node [right] {} (m-2-2)
    (m-2-1) edge node [below] {$\cong $} coordinate[midway] (b) (m-2-2);
  \end{tikzpicture}\]
and hence it is monic and thus injective. To show that it is surjective, we invert the isomorphism above and consider the commuting triangle
    \[\begin{tikzpicture}
    \node (x) at (-1,1) {$\colim PD$};
    \node (a) at (3,1) {$P\colim D$};
    \node (b) at (1,-1) {$\colim SD$};
    \draw[->] (x) to node[left] {} (a);
    \draw[->] (x) to node[right] {} (b);
    \draw[->] (a) to node[below] {} (b);
    \end{tikzpicture}\]
in order to rely on the concrete description of filtered colimits as quotients of the disjoint product. Then elements of $\colim SD$ are equivalence classes of terms that appear at some stage $SD(j)$: elements of $\colim PD\hookrightarrow \colim SD$ then are given by terms that commute generically with all function symbols of the theory already at some stage $SD(j)$, whereas $\colim PD\hookrightarrow \colim SD$ consists of those terms that commute generically with all function symbols of the theory at the colimit stage \ie after passing to equivalence classes.  Now, if such a term commutes generically with a given operation of the theory, it has to do so at some stage $j\in \JJ$. As our theory can be given by finitely many operations, a term commutes generically with all operations of the theory iff it commutes generically with this finite set of them. As the diagram is filtered, we may then first find for each of these finitely many operations a stage $j_i\in \JJ$  at which the term commutes with the $i$-th operation, and then find a stage $j\in \JJ$  above all of these. This shows that a term in $\colim PD\to \colim SD$ already has to commute generically with all operations of the theory at some stage $j\in \JJ$, \ie  that it is already in $\colim PD\hookrightarrow \colim SD$, as desired.
\end{proof}

Informally, the representing object can be then thought of as a ``a model freely generated with an inner automorphism''. Finding an explicit finite presentation for this object can be insightful, as it would give a ``universal term'' representing inner automorphisms in that theory just like the term $g x g^{-1}$ does for groups. We leave open the question of finding such an explicit finite presentation. Another potentially interesting further avenue is the following: in some cases, instead of a  ``universal term'' there seems to be a family of terms that capture all possible types of an inner automorphism (for instance, the results of~\cite{hofstrakarvonen:2d-isotropy} seem to suggest that this is the for racks and quandles). Do such situations correspond to the case where $\ZZ$ is familially representable? If so, is there a version of Theorem~\ref{thm:representability} capturing familial representability of $\ZZ$ in terms of the $2$-category $\CC_\ZZ$?  

Many of the results in this section, most notably Theorems~\ref{thm:connectedcolims} and~\ref{thm:positiveresults}(i) do not invoke inverses in the groups $\GG$ at all, and hence we expect that these results would go through even if $\GG$ was a copresheaf of monoids. Moreover, it seems that many of the other results that do invoke inverses and in particular the Peiffer identity could be written in an equivalent form not invoking inverses by systematically replacing equations of the form $gxg^{-1}=y$ with the equivalent $gx=yg$. However, achieving such additional generality would come at a cost as we would have had to first develop a monoid-version of crossed modules instead of referring to results from~\cite[Section 5]{funketal:crossedtoposes} in Section 2, and for relatively little gain: for us, the main crossed modules of interest come from isotropy, and the passage from isotropy groups to isotropy monoids has been studied relatively little as it seems that (extended) inner endomorphisms are less well behaved and less interesting than inner automorphisms, see \eg~\cite{bergman:inner}.

\section{Extending functors to $2$-functors}

While $\ZZ$ can be used to promote any locally small category (with small isotropy) into a (locally small) $2$-category $\CC_\ZZ$, this construction is not functorial in $\CC$. As the functor discussed in Theorem~\ref{thm:xmodto2cat} is full and faithful on 1 and $2$-cells, we can view this theorem as characterizing exactly the further data required to make an ordinary functor $\CC\to\DD$ into a $2$-functor $\CC_\ZZ\to\DD_\ZZ$, and similarly for natural transformations between them. 

While not directly related to the earlier results on (co)limits in $2$-categories of the form $\CC_\GG$, this section is also motivated by understanding such $2$-categories, this time in terms of functors between them. This section arose from studying in further detail when an ordinary functor $\CC\to\DD$ can be promoted into a  $2$-functor $\CC_\ZZ\to\DD_\ZZ$, whether uniquely or at least canonically. The first two results of this section stem from the observation that such a canonical extension always exists for the functor $\ZZ\colon \CC\to\Grp$ and its left adjoint if it exists. However, this holds in fact more generally: for any crossed module we have a canonical extension of $\GG\colon \CC\to\Grp$ to a $2$-functor $\CC_\GG\to \Grp_\ZZ$. Moreover, when $\GG$ has a left adjoint, the adjunction extends canonically to a 2-adjunction. We then conclude by exhibiting a class of functors that is closed under composition and for which the extension from $\CC\to\DD$ to $\CC_\ZZ\to\DD_\ZZ$ is unique. 

\begin{theorem}\label{thm:G_is_a_2functor} For any crossed module $(\GG \colon \CC\to\Grp,\gamma)$, the functor
$\GG$ extends canonically to a $2$-functor $\CC_\GG\to\Grp_\ZZ$. 
\end{theorem}

\begin{proof} The canonical isomorphism $\id[\Grp]\cong \ZZ_\Grp$ gives rise to an isomorphism $\sigma\colon\GG\to\ZZ\GG$. For an arbitrary $A\in\CC$, consider the square 
  \[\begin{tikzpicture}
    \matrix (m) [matrix of math nodes,row sep=2em,column sep=4em,minimum width=2em]
    {
     \GG(A) & \ZZ \GG(A) \\
      \Aut(A) & \Aut(\GG(A))  \\};
    \path[->]
    (m-1-1) edge node [left] {$\gamma$} (m-2-1)
           edge node [above] {$\sigma$} (m-1-2)
    (m-1-2) edge node [right] {$\delta$} (m-2-2)
    (m-2-1) edge node [below] {$\ZZ_\Grp$}  (m-2-2);
  \end{tikzpicture}\]
 The path via the top right corner sends $\alpha\in\GG(X)$ to the inner automorphism $\alpha\circ -\circ\alpha^{-1}$, whereas the path via the bottom left corner sends $\alpha$ to $\ZZ(\gamma(\alpha))$. These two agree by the Peiffer identity, giving us a morphism of crossed modules and hence a $2$-functor by Theorem~\ref{thm:xmodto2cat}. 
\end{proof}

\begin{example}
In general, this functor is not full on $2$-cells: for instance, in $\cat{Ab}_\ZZ$ there is a unique $2$-cell between $f$ and $-f$, and no other non-identity $2$-cells. However, the functor $\ZZ_{\cat{Ab}}$ is constant at $\mathbb{Z}_2$, so that between $\ZZ_{\cat{Ab}}(f)$ and $\ZZ_{\cat{Ab}}(g)$ there are always two $2$-cells for any parallel $f$ and $g$. 
\end{example}

\begin{theorem}\label{thm:2-adjoints} Let $(\GG \colon \CC\to\Grp,\gamma)$ be a crossed module for which the functor
$\GG$  has a left adjoint $L$. Then the adjunction $L\dashv \GG$ extends canonically to a strict 2-adjunction. In particular, for this extension $\GG$ preserves all two-dimensional limits that exist in $\CC$ and $L$ preserves all two-dimensional colimits in $\Grp$.
\end{theorem}

\begin{proof}
$\GG$ extends canonically to a $2$-functor by Theorem~\ref{thm:G_is_a_2functor}. Let us show that the left adjoint $L$ also extends canonically to a $2$-functor. Composing the canonical isomorphism $\tau\colon\ZZ\cong\id[\Grp]$ with the unit $\eta\colon \id[\Grp]\to \GG \circ L$ gives a natural transformation $\ZZ\to \GG\circ L$. For a group $H$, consider the square
  \[\begin{tikzpicture}
    \matrix (m) [matrix of math nodes,row sep=2em,column sep=4em,minimum width=2em]
    {
     \ZZ (H)\cong H & \GG L(H) \\
      \Aut(H) & \Aut(LH)  \\};
    \path[->]
    (m-1-1) edge node [left] {$\delta$} (m-2-1)
           edge node [above] {$\eta$} (m-1-2)
    (m-1-2) edge node [right] {$\gamma $} (m-2-2)
    (m-2-1) edge node [below] {$L$}  (m-2-2);
  \end{tikzpicture}\]
Going via the bottom left corner sends $\alpha\in\ZZ (H)$ to $L(\alpha_{\id} )\colon LH\to LH$ and going via the top right corner sends $\alpha$ to $\gamma_{LH}\eta_H\tau_H(\alpha)\colon LH\to LH$.  By adjointness, it suffices to compare the corresponding morphisms $H\to \GG L(H)$. Saying that these are equal amounts to claiming that boundary of
\[\begin{tikzpicture}[yscale=1.75,xscale=1.75]
\node (a) at (-3,0) {$H$};
\node (b0) at (0,1.5) {$\GG L(H)$};
\node (b1) at (0,.5) {$H$};
\node (b2) at (0,-.5) {$\GG L(H)$};
\node (b3) at (0,-1.5) {$\GG L(H)$};
\node (c) at (3,0) {$\GG L(H)$};
\draw[->,out=60,in=180] (a) to node[above] {$\eta_H$} (b0);
\draw[->] (a) to node[above] {$\alpha_{\id}$} (b1);
\draw[->] (a) to node[above] {$\eta_H$} (b2);
\draw[->,out=-60,in=180] (a) to node[below] {$\eta_H$} (b3);
\draw[->,out=0,in=120] (b0) to node[right] {\quad$\GG L(\alpha_{\id})$} (c);
\draw[->] (b1) to node[above] {$\eta_H$} (c);
\draw[->] (b2) to node[above] {$\alpha_{\eta_H}$} (c);
\draw[->,out=0,in=-120] (b3) to node[below] {$\qquad\qquad\qquad\GG(\gamma (\eta_H\tau_H(\alpha)))$} (c);
\draw[gray,draw=none] (b0) to node{(i)} (b1);
\draw[gray,draw=none] (b1) to node{(ii)} (b2);
\draw[gray,draw=none] (b2) to node{(iii)} (b3);
\end{tikzpicture}\]
commutes, which will follow once we show that the regions inside of it commute. Region (i) commutes by naturality of $\eta$ and (ii) by naturality of $\alpha$. Let us consider region (iii): the top path first sends $h\in H$ to $\eta_H(h)$ and then conjugates it with $\eta_H\tau_H(\alpha)$, resulting in $(\eta_H\tau_H(\alpha))\eta_H(g)(\eta_H\tau_H(\alpha))^{-1}$, whereas the bottom path sends $h$ to $\GG(\gamma(\eta_H\tau_H(\alpha)))\eta_H(h)$. These two are equal by the Peiffer identity, so that we can promote $L$ into a morphism of crossed modules and hence to a $2$-functor by Theorem~\ref{thm:xmodto2cat}. 

It remains to check that the unit and counit are $2$-natural, which by Theorem~\ref{thm:xmodto2cat} amounts to checking that they are $2$-cells between maps of crossed modules. The $2$-naturality of the unit $\eta\colon \id[\Grp]\to\GG L$ follows from the commutativity of
  \[\begin{tikzpicture}[yscale=1,xscale=1.75]
  \node (a) at (-2,2) {$\ZZ$};
  \node (b) at (2,2) {$\ZZ$};
  \node (c) at (2,-2) {$\ZZ\GG L$};
  \node (a1) at (0,0) {$\ZZ\GG L$};
  \node (a2) at (-2,-2) {$\id[\Grp]$};
  \node (a3) at (0,-2) {$\GG L$};
  \draw[->] (a) to node[above] {$\id$} (b);
  \draw[->] (a) to node[above] {$\ZZ\eta$} (a1);
  \draw[->] (a) to node[left] {$\tau$} (a2);
  \draw[->] (a1) to node[left] {$\tau$} (a3);
  \draw[->] (a1) to node[right] {$\id$} (c);
  \draw[->] (a2) to node[above] {$\eta$} (a3);
  \draw[->] (a3) to node[above] {$\tau^{-1}$} (c);
  \draw[->] (b) to node[right] {$\ZZ \eta $} (c);
  \end{tikzpicture}\]
and the counit $\epsilon\colon L\GG\to \id[\CC]$  is $2$-natural since
  \[\begin{tikzpicture}[yscale=1,xscale=1.75]
  \node (a) at (-2,2) {$\GG$};
  \node (a1) at (0,2) {$\ZZ\GG$};
  \node (a2) at (2,2) {$\GG$};
  \node (b) at (2,0) {$\GG L\GG$};
  \node (c) at (2,-2) {$\GG$};
  \draw[->] (a) to node[below] {$\id$} (c);
  \draw[->] (a) to node[above] {$\tau^{-1}$} (a1);
 \draw[->] (a1) to node[above] {$\tau$} (a2);
   \draw[->] (a2) to node[right] {$\eta\GG$} (b);
  \draw[->] (b) to node[right] {$\GG \epsilon$} (c);
  \end{tikzpicture}\]
commutes.
\end{proof}

We conclude by observing that, under certain conditions, there is at most one extension of $F\colon\CC\to\DD$ to a $2$-functor $\CC_\ZZ\to\DD_\ZZ$.

\begin{proposition}
Let $F\colon\CC\to\DD$ be a left adjoint with a faithful right adjoint $G$. Then there is at most one extension of $F$ to a $2$-functor $\CC_\ZZ\to\DD_\ZZ$.

Moreover, such an extension exists iff $F$ induces a functor of isotropy quotients, \ie if whenever there is a $2$-cell $f\to g$, there is a $2$-cell $F(f)\to F(g)$. 
\end{proposition}

\begin{proof}
We prove the first claim by showing that for any $\alpha\in \ZZ(A)$, there is at most one $\beta\in \ZZ_\DD F(A)$ such that $\beta_{Ff}=F(\alpha_f)$ for every $f\colon A\to B$ in $\CC$. Indeed, assume $\beta,\beta'\in \ZZ_\DD F(A)$ satisfy this condition, and consider $g\colon F(A)\to B$ in $\DD$. To show that $\beta_g=\beta'_g$, consider the unique $\bar{g}\colon A\to G(B)$ making the triangle
  \[\begin{tikzpicture}
  \node (a) at (0,0) {$F(A)$};
  \node (b) at (3,1) {$FG(B)$};
  \node (c) at (3,-1) {$B$};
  \draw[->] (a) to node[above] {$F(\bar{g})$} (b);
  \draw[->] (a) to node[below] {$g$} (c);
  \draw[->] (b) to node[right] {$\epsilon_B$} (c);
  \end{tikzpicture}\]
  commute. Now, the constraints on $\beta,\beta'$ imply that $\beta_{F\bar{g}}=\beta_{F\bar{g}}'$. As $G$ is faithful, $\epsilon$ being pointwise epic. Thus the square
    \[\begin{tikzpicture}
    \matrix (m) [matrix of math nodes,row sep=2em,column sep=4em,minimum width=2em]
    {
     FG(B) & FG(B)\\
      B & B  \\};
    \path[->]
    (m-1-1) edge node [left] {$\epsilon_B$} (m-2-1)
           edge node [above] {$\beta_{F\bar{g}}=\beta_{F\bar{g}}'$} (m-1-2)
    (m-1-2) edge node [right] {$\epsilon_B$} (m-2-2)
    (m-2-1) edge node [below] {$\beta_g$}  (m-2-2);
  \end{tikzpicture}\]
forces $\beta_g=\beta_g'$, as desired. 

We now move to the next claim: given $\alpha\in \ZZ(A)$ we first show that there exists $\beta\in \ZZ_\DD F(A)$ such that $\beta_{Ff}=F(\alpha_f)$. For $g\colon F(A)\to B$, we define $\beta_g$ as the unique map making the square 

    \[\begin{tikzpicture}
    \matrix (m) [matrix of math nodes,row sep=2em,column sep=4em,minimum width=2em]
    {
     FG(B) & FG(B)\\
      B & B  \\};
    \path[->]
    (m-1-1) edge node [left] {$\epsilon_B$} (m-2-1)
           edge node [above] {$F(\alpha_{\bar{g}})$} (m-1-2)
    (m-1-2) edge node [right] {$\epsilon_B$} (m-2-2)
    (m-2-1) edge node [below] {$\beta_g$}  (m-2-2);
  \end{tikzpicture}\]
 Such an arrow exists since $\ZZ(\bar{g})(\alpha)$ defines a $2$-cell $\id[G(B)]\to \alpha_{\bar{g}}$ and hence by assumption there is a $2$-cell $\id[FG(B)]\to F(\alpha_{\bar{g}})$ so that $F(\alpha_{\bar{g}})$ is part of some isotropy (so that such $\beta_g$ indeed exists). Moreover, $\beta_g$ is uniquely defined by this as $\epsilon_B$ is epic. A straightforward calculation shows that $\beta$ as so defined defines an element of $\ZZ_\DD F(A)$. We now define $\sigma_A(\alpha)=\beta$, whence it follows that $\sigma_A$ defines a homomorphism $\ZZ_\CC A\to \ZZ_\DD F(A)$. To see that these homomorphisms are natural in $A$, consider the square 
     \[\begin{tikzpicture}
    \matrix (m) [matrix of math nodes,row sep=2em,column sep=4em,minimum width=2em]
    {
     \ZZ_\CC(A) & \ZZ_\DD F(A)\\
      \ZZ_\CC (B) &  \ZZ_\DD F(B)  \\};
    \path[->]
    (m-1-1) edge node [left] {$\ZZ_\CC(f)$} (m-2-1)
           edge node [above] {$\sigma_A$} (m-1-2)
    (m-1-2) edge node [right] {$\ZZ_\DD F(f)$} (m-2-2)
    (m-2-1) edge node [below] {$\sigma_B$}  (m-2-2);
  \end{tikzpicture}\]
 and observe that sending $\alpha\in \ZZ_\CC(A)$ along either path results in $\beta\in \ZZ_\DD F(B)$ satisfying 
 $\beta_{Fg}=F(\alpha_{gf})$ for any $g\colon B\to C$, so that by uniqueness of such $\beta$ the square commutes.
\end{proof}

\bibliographystyle{plain}
\bibliography{hnn}

\end{document}